\documentclass[12pt,a4paper]{article}
\usepackage{amsmath}
\usepackage{amsfonts}
\usepackage{amssymb}
\usepackage{mathrsfs}
\usepackage[english]{babel}
\usepackage{amsmath, amssymb, amsthm}
\usepackage[dvips,final]{epsfig}
\usepackage[dvips,final]{graphicx}
\usepackage{amsmath}
\usepackage{amsthm}
\usepackage{enumerate}
\usepackage{setspace}
\usepackage{amsfonts}
\usepackage{empheq}

\usepackage{color}
\definecolor{erica}{rgb}{0.69, 0.19, 0.38}
\definecolor{anna}{rgb}{0.13, 0.55, 0.13}
\definecolor{bv}{rgb}{0.54, 0.17, 0.89}
\definecolor{lightgray}{rgb}{0.83, 0.83, 0.83}

\newtheorem{theorem}{\bf Theorem}[section]
\newtheorem{lemma}[theorem]{\bf Lemma}

\newtheorem{definition}[theorem]{\bf Definition}
\newtheorem{remark}[theorem]{\bf Remark}
\newtheorem{proposition}[theorem]{ Proposition}

\def \k {{\kappa}}

\def \d {{\delta}}

\def \L {\mathscr{L}}
\def \K {\mathscr{K}}

\def \R {{\mathbb {R}}}
\def \N {{\mathbb {N}}}

\def \H {\mathcal{Q}}
\def \x {{\xi}}
\def \e {\varepsilon}

\def \t {{\tau}}

\def \z {{\zeta}}

\def \tilde {\widetilde}

\def\p{\partial}
\def \P {{\cal{P}}}

\def \rnn {{\mathbb {R}}^{N+1}}

\def \OO {\mathbb{O}}
\def \P {\mathcal{P}}
\def \r {\varrho}
\def \ip {\frac{1}{p}}
\def \meas {{\text{\rm meas}}}

\usepackage[english]{babel}
\begin{document}
\title{\bf Pointwise estimates for degenerate \\ Kolmogorov equations with $L^p$-source term}

\author{\sc{Erica Ipocoana}\thanks{Dipartimento di Matematica Pura e Applicata, Universit\`{a} di Modena e Reggio Emilia, via Campi 213/b, 41115 Modena (Italy). E-mail: erica.ipocoana@unipr.it}  \;\textrm{and} 
\sc{Annalaura Rebucci}\thanks{Dipartimento di Scienze Matematiche, Fisiche e Informatiche,
Universit\`{a} di Parma, Parco Area delle Scienze 53/A, 43124 Parma (Italy). E-mail: annalaura.rebucci@unipr.it}}

\maketitle
\begin{abstract}
The aim of this paper is to establish new pointwise regularity results for solutions to degenerate second order partial differential equations with a Kolmogorov-type operator of the form
\begin{equation*} 
\L :=\sum_{i,j=1}^m \partial^2_{x_i x_j } +\sum_{i,j=1}^N b_{ij}x_j\partial_{x_i}-\partial_t,
\end{equation*}
where $(x,t) \in \R^{N+1}$, $1 \le m \le N$ and the matrix $B:=(b_{ij})_{i,j=1,\ldots,N}$ has real constant entries. In particular, we show that if the modulus of $L^p$-mean oscillation of $\L u$ at the origin is Dini, then the origin is a Lebesgue point of continuity in $L^p$ average for the second order derivatives $\partial^2_{x_i x_j} u$, $i,j=1,\ldots,m$, and the Lie derivative $\left(\sum_{i,j=1}^N b_{ij}x_j\partial_{x_i}-\partial_t\right)u$. Moreover, we are able to provide a Taylor-type expansion up to second order with estimate of the rest in $L^p$ norm. The proof is based on decay estimates, which we achieve by contradiction, blow-up and compactness results.
\end{abstract}


\setcounter{equation}{0} \setcounter{theorem}{0}
\section{Introduction}
In this paper, we study the pointwise regularity of solutions to the following Cauchy problem
\begin{equation}\label{pbm}
    \begin{cases}
    		\L u=f \quad {\text {in}}\quad  \H_1^- \\
    		f\in L^p(\H_r^-) \quad {\text {and} } \quad f(0)=0,
    \end{cases}
\end{equation}
where $\H_r^- = B_r \times (-r^2,0)$ is the past cylinder defined through the open ball $B_r = \lbrace x \in \R^N : |x|_K \leq r \rbrace$ and $|\cdot|_K$ is the semi-norm, due to the nature of operator $\L$, defined in \eqref{norm-def-form}.\\
We suppose here that $1<p<\infty$ and that the origin $0=(0,0)$ is a Lebesgue point of $f$, so that we are able to define $f(0)$ if needed.\\
We denote by $\L$ the second order linear differential operators of Kolmogorov type of the form 
\begin{equation} \label{e-Kolm-costc}
\L :=\sum_{i,j=1}^m \partial^2_{x_i x_j } +\sum_{i,j=1}^N b_{ij}x_j\partial_{x_i}-\partial_t,
\end{equation}
where $(x,t) \in \R^{N+1}$, and $1 \le m \le N$. 
The matrix $B:=(b_{ij})_{i,j=1,\ldots,N}$ has real constant entries. \\
It is natural to place operator $\L$ in the framework of H\"{o}rmander's theory. More precisely, let us set
\begin{eqnarray*}
X_i:=\partial_{x_i}, \quad i=1,\ldots,m, \quad Y := \sum_{i,j=1}^N b_{ij}x_j\partial_{x_i}-\partial_t = \langle B x, D \rangle-\partial_t,
\end{eqnarray*}
where $\langle \cdot, \cdot \rangle$ and $D$ denote the inner product and the gradient in $\R^{N}$, respectively. Then the operator $\L$ can be written as 
\begin{eqnarray*}
\L =\sum_{i=1}^m X_i^2+Y.
\end{eqnarray*}
It is known that under the H\"{o}rmander's condition (see \cite{Hormander})
\begin{equation*}\label{e-Horm}
{\rm rank\ Lie}\left(X_1,\dots,X_m,Y\right)(x,t) =N+1,\qquad \forall \, (x,t) \in \R^{N+1},
\end{equation*}
$\L$ is hypoelliptic, namely that every distributional solution $u$ to $\L u = f$ defined in some open set $\Omega \subset \R^{N+1}$ belongs to $C^\infty(\Omega)$, and it is a classical solution to $\L u = f$, whenever $f \in C^\infty(\Omega)$.\\
In the sequel, we will assume the following hypothesis on the Kolmogorov operator $\L$.
\begin{itemize}
\item[{\bf[H.1]}] $\L$ is hypoelliptic and $\delta_r$-homogeneous of degree two with respect to some dilations group $(\delta_r)_{r>0}$ in $\R^{N+1}$ (see \eqref{e-dilations} below).
\end{itemize}
We remark that, if $\L$ is uniformly parabolic (i.e. $m=N$ and $B\equiv \mathbb{O}$), then assumption $\bf[H.1]$ is clearly satisfied. In fact, in this case operator $\L$ is simply the heat operator, which is known to be hypoelliptic. However, in this note we are mainly interested in the genuinely degenerate setting.
\medskip
\\
It is known that the natural geometry when studying operator $\L$ is determined by a suitable homogeneous Lie group structure on $\R^{N+1}$. More precisely, as first observed by Lanconelli and Polidoro in \cite{LanconelliPolidoro}, operator $\L$ is invariant with respect to left translation in the group $\mathbb{K}=(\mathbb{R}^{N+1},\circ)$, where the group law is defined by 
\begin{equation}
	\label{grouplaw}
	(x,t) \circ (\xi, \tau) = (\xi + E(\tau) x, t + \tau ), \hspace{5mm} (x,t),
	(\xi, \tau) \in \R^{N+1},
\end{equation}
and
\begin{equation}\label{exp}
	E(s) = \exp (-s B), \qquad s \in \R.
\end{equation} 
Then $\mathbb{K}$ is a non-commutative group with zero element $(0,0)$ and inverse
\begin{equation*}
(x,t)^{-1} = (-E(-t)x,-t).
\end{equation*}
For a given $\zeta \in \R^{N+1}$ we denote by $\ell_{\z}$ the left traslation on $\mathbb{K}=(\R^{N+1},\circ)$ defined as follows
\begin{equation*}
	\ell_{\z}: \R^{N+1} \rightarrow \R^{N+1}, \quad \ell_{\z} (z) = \z \circ z.
\end{equation*}
Then the operator $\K$ is left invariant with respect to the Lie product $\circ$, that is
\begin{equation*}
	\label{ell}
    \L \circ \ell_{\z} = \ell_{\z} \circ \L \qquad {\rm \textit{or, equivalently,}} 
    \qquad \L\left( u( \z \circ z) \right)  = \left( \L u \right) \left( \z \circ z \right),
\end{equation*}
for every $u$ sufficiently smooth.

We explicitly remark that, by Propositions 2.1 and 2.2 in \cite{LanconelliPolidoro}, hypothesis $\bf[H.1]$ is equivalent to assume that, for some basis on $\R^N$, the matrix $B$ takes the following form
\begin{equation*} \label{B_0}
 			   B_0 = \begin{pmatrix}
 			       \OO &  \OO & \ldots & \OO & \OO \\
 			       B_1 & \OO &  \ldots & \OO & \OO\\
                   \OO & B_{2}  & \ldots & \OO & \OO \\
                   \vdots & \vdots  & \ddots & \vdots & \vdots  \\
                   \OO & \OO & \ldots & B_{\k} & \OO
   			 \end{pmatrix}
\end{equation*}
where every block $B_j$ is a $m_{j} \times m_{j-1}$ matrix of rank $m_j$ with $j = 1, 2, \ldots, \k$. Moreover, the $m_j$s are positive integers such that
\begin{eqnarray*}
m_0 \geq m_1 \geq \ldots \geq m_\kappa \geq 1, \quad \textrm{and} \quad m_0+m_1+\ldots+m_\kappa=N.
\end{eqnarray*}
We agree to let $m_0 :=m$ to have a consistent notation, moreover $\OO$ denotes a block matrix whose entries are zeros. In the sequel we shall always assume that $B$ has the canonical form \eqref{B_0}.

In this case the dilation is defined for every positive $r$ as
\begin{equation}\label{e-dilations}
    \delta_r :=\textrm{diag}(r I_{m}, r^3 I_{m_1}, \ldots, r^{2\kappa+1}I_{m_\kappa},r^2),
\end{equation}
where $I_k$, $k\in\N$, is the $k$-dimensional unit matrix, and the second assertion in assumption ${\bf[H.1]}$ reads as follows
\begin{equation}
	\label{Ginv}
      	 \L \left( u \circ \delta_r \right) = r^2 \delta_r \left( \L u \right), \quad \text{for every} \quad r>0.
\end{equation}
It is also useful to denote by $\left(\delta_r^0 \right)_{r > 0}$ the family of spatial dilations defined as 
\begin{equation}
	\label{fam-dil-space}
	\delta_r^0 = \text{diag} ( r I_{m} , r^3 I_{m_1}, \ldots, r^{2\k+1} I_{m_\k} ) 
   		 \quad {\rm for \, every} \, \, r > 0.
\end{equation}

The integer numbers
\begin{equation}
	\label{hom-dim}
	Q := m_{0} + 3m_{1} + \ldots + (2\k+1) m_{k},  \quad \text{and} \quad Q + 2
\end{equation}
will be named \emph{homogeneous dimension of $\R^{N}$ with respect to $(\delta_r^0)_{r>0}$}, and \emph{homogeneous dimension of $\R^{N+1}$ with respect to $(\delta_r)_{r >0}$}, because we have that
\begin{equation*}
	\det \, \delta_r^0 = r^{Q } \quad \text{and} \quad \det \, \delta_r = r^{Q + 2} \qquad \text{for every} \ r > 0.
\end{equation*}
We next introduce a homogeneous norm of degree $1$ with respect to the dilations 
$(\d_{r})_{r>0}$ and a corresponding quasi-distance which is invariant with respect to the group operation in \eqref{grouplaw}.
We first rewrite the matrix $\delta_r$ with the equivalent notation
\begin{equation}\label{e-dilations-cp}
    \delta_r :=\textrm{diag}(r^{\alpha_1}, \ldots, r^{\alpha_N},r^2),
\end{equation}
where $\alpha_1, \dots, \alpha_{m_0} =1, \alpha_{m_0+1}, \dots, \alpha_{m_0+m_1} = 3, \alpha_{N-m_\kappa}, \dots, \alpha_N = 2 \kappa + 1$. 
\begin{definition}
	\label{norm-def}
    For every $(x,t) \in \R^{N+1}$ we set
 \begin{equation}\label{norm-def-form}
		\Vert (x,t) \Vert_K =|t|^{\frac{1}{2}}+|x|, \quad |x|_K=\sum_{j=1}^N |x_j|^{\frac{1}{\alpha_j}}
	\end{equation}
	where the exponents $\alpha_j$, for $j=1,\ldots,N$, were introduced in \eqref{e-dilations-cp}
\end{definition}
Note that the semi-norm is homogeneous of degree $1$ with respect to the family of dilations $(\delta_r)_{r>0}$, namely $ \|\delta_r (x,t) \|_K =r \| (x,t)\|_K $ for every $r>0$ and $(x,t)\in\R^{N+1}$. In addition, the following pseudo-triangular inequality holds: for every bounded set $H \subset \R^{N+1}$ there exists a positive constant ${\bf c}_H$ such that
\begin{equation}\label{e-ps.tr.in}
 \|(x,t)^{-1}\|_K \le {\bf c}_H  \| (x,t) \|_K, \qquad  
 \|(x,t) \circ (\x,\t) \|_K \le {\bf c}_H  (\| (x,t) \|_K + \| (\x,\t) \|_K), 
\end{equation}
for every $(x,t), (\x,\t) \in H$. We then define the {\it quasi-distance} $d_K$ by setting
\begin{equation}\label{e-ps.dist}
    d_K( (x,t) ,(\x,\t)):= \|(\x,\t)^{-1}\circ (x,t)\|_K, \qquad (x,t), (\x,\t) \in \R^{N+1},
\end{equation}
and the {\it unit cylinder}
\begin{equation*}
\H_1=\lbrace (x,t) \in \R^{N+1} \mid |x|_K < 1, \quad t \in (-1,0) \rbrace.
\end{equation*}
For every $(x_0,t_0) \in \R^{N+1}$ and $r>0$, we set
\begin{equation*}
\H_r(x_0,t_0):=z_0 \circ \delta_r(\H_1)=\lbrace (x,t) \in \R^{N+1} \mid (x,t)=(x_0,t_0) \circ \delta_r(\x,\t),(\x,\t) \in \H_1 \rbrace.
\end{equation*}
We also observe that the Lebesgue measure is invariant with respect to the translation group 
	associated to $\L$, since $\det E(s) = e^{s \hspace{1mm} \text{\rm trace} \,
		B} = 1$. Moreover, we have 
	\begin{equation*}
		\meas \left( \H_r(x_0,t_0) \right) = r^{Q+2} \meas \left( \H_1(x_0,t_0) \right), \qquad \forall
		\ r > 0, (x_0,t_0) \in \R^{N+1}.
	\end{equation*}
Finally, we recall that, under the the hypothesis of hypoellipticity, H\"{o}rmander in \cite{Hormander} constructed the fundamental solution of $\L$ as
\begin{equation*} \label{eq-Gamma0}
	\Gamma (z,\zeta) = \Gamma (\zeta^{-1} \circ z, 0 ), \hspace{4mm} 
	\forall z, \zeta \in \R^{N+1}, \hspace{1mm} z \ne \zeta,
\end{equation*}
where
\begin{equation} \label{eq-Gamma0-b}
	\Gamma ( (x,t), (0,0)) = \begin{cases}
		\frac{(4 \pi)^{-\frac{N}{2}}}{\sqrt{\text{det} C(t)}} \exp \left( - 
		\frac{1}{4} \langle C^{-1} (t) x, x \rangle - t \, tr (B) \right), \hspace{3mm} & 
		\text{if} \hspace{1mm} t > 0, \\
		0, & \text{if} \hspace{1mm} t < 0,
	\end{cases}
\end{equation}
and
\begin{equation*} 	\label{c}
 C(t) = \int_{0}^{t} E(s) \, \begin{pmatrix}
 			       \mathbb{I}_m & \OO  \\ \OO & \OO 
   			 \end{pmatrix} \, E^{T}(s) \, ds.
\end{equation*}
Note that the first condition of assumption $\bf{[H.1]}$ implies that $C(t)$ is strictly positive for every $t>0$ (see \cite{LanconelliPolidoro}) and therefore $\Gamma$ in \eqref{eq-Gamma0-b} is well-defined.
\\
We now briefly discuss the applicative and theoretical interest in the study of operator $\L$. A simple meaningful example is the operator introduced by Kolmogorov in \cite{Kolmo}, defined for $(x,t) = (v,y,t) \in \R^{m} \times \R^m \times \R$ as follows
\begin{equation} \label{e-Kolm-costc0}
   \K := \sum_{j=1}^m \partial^{2}_{x_j} - \sum_{j=1}^m x_j \partial_{x_{m+j}}-\partial_t 
   = \Delta_v - \langle v, D_y \rangle - \partial_t.
\end{equation} 
The operator $\K$ can be written in the form \eqref{e-Kolm-costc} with $\kappa = 1, m_1 = m$, and 
   \begin{equation} \label{B_K}
 			   B = \begin{pmatrix}
 			       \OO & \OO  \\ - I_m & \OO 
   			 \end{pmatrix}
\end{equation}
The operator defined in \eqref{e-Kolm-costc0} arises in several areas of application of PDEs. In particular, in kinetic theory the density $u$ of particles, with velocity $v$ and position $y$ at time $t$, satisfies equation $\K u = 0$. In this setting, the Lie group associated to Kolmogorov operator has a quite natural intepretation. In fact, the composition law \eqref{grouplaw} agrees with the \emph{Galilean} change of variables 
\begin{equation*}
    (v,y,t) \circ (v_{0}, y_{0}, t_{0}) = (v_{0} + v, y_{0} + y + t v_{0}, t_{0} + t), 
    \qquad (v,y,t), (v_{0}, y_{0}, t_{0}) \in \R^{2m+1}.
\end{equation*}
It is easy to see that $\K$ is invariant with respect to the above change of variables. Specifically, if $w(v,y,t) = u (v_{0} + v, y_{0} + y + t v_{0}, t_{0} + t)$ and $g(v,y,t) = f(v_{0} + v, y_{0} + y + t v_{0}, t_{0} + t)$, then 
\begin{equation*}
	\K u = f \quad \iff \quad \K w = g \quad \text{for  every} \quad (v_{0}, y_{0}, t_{0}) \in \R^{2m+1}.
\end{equation*}
As the matrix $B$ in \eqref{B_K} is in the form \eqref{B_0}, $\K$ is invariant with respect to the dilatation $\delta_r(v,y,t) := (r v, r^3 y, r^2 t)$. Note that the dilatation acts as the usual parabolic scaling with respect to the variable $v$ and $t$, as the operator is uniformly parabolic with respect to these variables. The term $r^3$ in front of $y$ is due to the fact that the velocity $v$ is the derivative of the position $y$ with respect to time $t$. For a more comprehensive description of operator $\L$, and of its applications, we refer to the survey article \cite{AnceschiPolidoro} by Anceschi and Polidoro and to its bibliography. 
\medskip
\\
The aim of this paper is to study pointwise regularity of solutions to problem \eqref{pbm} for Kolmogorov equations with right hand side in $L^p$. This work may be seen as a generalisation of \cite{Monneau} and \cite{LindMonneau} where this kind of results are obtained for elliptic and parabolic equations respectively. However, up to our knowledge, the case of Kolmogorov type operators has not been investigated.\\ The main difficulty with respect to the previous literature lies in the fact that the regularity properties of the Kolmogorov equations on $\R^{N+1}$ depend strongly on the geometric Lie group structure introduced in \eqref{grouplaw}. In particular this reflects on the family of dilations we consider. Furthermore, according to \eqref{e-Kolm-costc}, we here take into account also the case where $m<N$ and therefore $\L$ is strongly degenerate. We emphasize that when $m=N$ and $B\equiv \mathbb{O}$, our result restores the one contained in \cite{LindMonneau}.\\
We remark that several Schauder type estimates have been proved e.g. by \v{S}atyro \cite{satyro}, Manfredini \cite{Manfredini} in the case of dilation invariant operators, Di Francesco et al. for dilation non-invariant operators in \cite{DiFrancescoPolidoro} and Polidoro et al. in \cite{prs}. In particular we note that in \cite{prs} the right hand side $f$ is Dini continuous, differently from the previous literature where $f$ was considered H\"{o}lder continuous. 
Moreover, we recall the works by Lunardi \cite{lunardi}, Lorenzi \cite{lorenzi} and Priola \cite{priola} in the framework of semigroup theory. \\
In \cite{pascuccipolidoro} and then in \cite{cintipascuccipolidoro}, a pointwise estimate for the weak solutions to Kolmogorov equations with right hand side equal to zero is proved. In order to do so, the  authors adapt the Moser iterative method to the non Euclidean framework of the (homogeneus and non-homogeneus, respectively) Lie groups. Finally, the regularity of strong solutions to the Cauchy-Dirichlet and obstacle problem for a class of Kolmogorov-type operators was studied in \cite{nystrom} using a blow-up technique.\\

While the previous results were derived assuming a modulus of continuity defined on some open set, we here introduce a \textit{pointwise modulus of mean oscillation}.\\
More precisely, following \cite{Monneau}, for $p\in (1,+\infty)$, we define the following \textit{modulus of $L^p$-mean oscillation} for the function $f$ \textit{at the origin} as
\begin{equation}\label{mmo}
\tilde {\omega}(f,r):=\inf_{c\in \R} \Big(\frac{1}{|\H_r^-|}\int_{\H_r^-}|f(x,t)-c|^p\Big)^{1\over p}.
\end{equation}
Moreover we set
\begin{equation}\label{ntilde}
\tilde {N}(u,r):=\inf_{P\in\mathcal{\tilde{P}}}  \Big(\frac{1}{r^{Q+2+2p}}\int_{\H_r^-}|u-P|^p\Big)^{1\over p},
\end{equation}
where $\tilde{\P}$ denotes the set of polynomials of degree less than or equal to two in $x_1 \ldots x_m  $ and less or equal to one in time and $Q$ is the homogeneous dimension defined in \eqref{hom-dim}. This paper is devoted to prove the following Theorem.

\begin{theorem}\label{mainthm}
Let $p\in(1,\infty)$. Then there exist $\alpha \in (0,1]$ and constants $r_* \in (0,1]$, $C >0,$ such that the following holds. In particular, if $u \in L^p(\H_1^-)$ satisfies \eqref{pbm} with the associated $\tilde{\omega}$ defined in \eqref{mmo}, then we have
\begin{itemize}
\item[i)]Pointwise BMO estimate
\begin{eqnarray}
\sup_{r \in (0,1]}\tilde{N}(u,r)\leq C \left\lbrace \left( \int_{\H_1^-}|u|^p\right)^{\ip}+\left(\int_{\H_1^-}|f|^p\right)^\ip +\sup_{r \in (0,1]}\tilde{\omega}(r)\right\rbrace.
\end{eqnarray}
\item[ii)]Pointwise VMO estimate
\begin{eqnarray}
\left( \tilde{\omega}(r)\rightarrow 0 \textrm{ as } r\rightarrow 0^+ \right) \quad \Rightarrow \quad
\left( \tilde{N}(u,r)\rightarrow 0 \textrm{ as } r\rightarrow 0^+ \right). 
\end{eqnarray}
\item[iii)] Pointwise control on the solution\\
If $\tilde{\omega}$ is Dini, then $\tilde{N}(u,\cdot)$ is Dini. Moreover, there exists a polynomial $P_0$ which is a solution to equation $\L P_0 = 0$ of degree less than or equal to two in $x_1\ldots x_m$ and of degree less than or equal to one in time, such that for every $r \in (0,r_*]$ there holds 

\begin{equation} \label{est-thm}
\begin{split}
&\left(\frac{1}{|\H_r^-|}\int_{\H_r^-}\Bigg|\frac{u(x,t)-P_0(x,t)}{r^2}\Bigg|^p\right)^{\ip}\\
&\qquad\qquad\qquad \leq C\left\lbrace \tilde{M}_0 \left(\frac{{4r}}{\lambda}\right)^{\beta}+\int_0^{4r}\frac{\tilde{\omega}(s)}{s}ds + r^{\beta}\int_{4r}^1\frac{\tilde{\omega}(s)}{s^{1+\beta}}ds \right\rbrace.
\end{split}
\end{equation}

and
\begin{eqnarray*}
P_0(x,t) = a+b\cdot x +\frac{1}{2} {}^T x \cdot c \cdot x + d \ t
\end{eqnarray*}
with
\begin{equation*}
|a|+ |b|+ |c| + |d| \leq C \tilde{M}_0 \quad \textrm{ and } \quad \tilde{M}_0 = \int_0^1 \frac{\tilde{\omega}(s)}{s}ds + \left( \int_{\H_1^-} |u|^p\right)^\ip+ \left( \int_{\H_1^-}|f|^p\right)^\ip.
\end{equation*}
\end{itemize}
\end{theorem}

We observe that a simple consequence of Theorem \ref{mainthm} is that the second order derivatives $\partial^2_{x_i x_j} u$, $i,j=1,\ldots,m$, and the Lie derivative $Y u$ are H\"{o}lder continuous in some open set $\Omega \subset \R^{N+1}$, when $\L u$ is H\"{o}lder continuous with respect to the distance introduced in \eqref{e-ps.dist}. Moreover, let us remark that Theorem \ref{mainthm} provides us with a Taylor-type expansion up to second order with estimate of the rest in $L^p$ norm. 
We finally emphasize that our result is completely pointwise, which does not seem to be usual when dealing with Kolmogorov-type operators. \smallskip\\
The structure of the paper is the following. In Section \ref{notation} we introduce some notation  which is used throughout the paper. Some general control results are contained in Section \ref{prel-results}, some from the literature and a Caccioppoli-type estimate we proved ad hoc for our problem. Finally, Section \ref{mainsection} is devoted to prove our main result Theorem \ref{mainthm}. In particular, for sake of semplicity, we first derive some preliminary estimates in Section \ref{prel-est} in order to finally give a shorter proof of Theorem \ref{mainthm} in Section \ref{sectionproof}.

\setcounter{equation}{0}\setcounter{theorem}{0}
\section{Notation}
\label{notation}
In this section we introduce some further notation which will be used throughout the paper.\\
We define the following classes of polynomials.
\begin{align} \nonumber
\mathcal{\tilde{P}} = & \left\lbrace P : \textrm{ polynomials of degree less or equal to two in } \; x_1 \ldots x_m  \right. \\ \label{ptil}
& \left. \textrm{  and less or equal to one in time} \right\rbrace .
\end{align}
\begin{equation} \label{p}
\P \colon= \Big\{  P \in \tilde{\P}: \L P = 0 \Big\}.
\end{equation}

\begin{equation} \label{pc}
\P_c \colon= \Big\{  P \in \tilde{\P}: \L P = c \Big\}.
\end{equation}

\noindent In particular we take $P_*$ such that $\L P_*=1$ and set $\mathcal{P}_{c} =c P_*+\mathcal{P}$. \\
Owing to \eqref{mmo}, let $c_r$ be a constant such that
\begin{equation}\label{mmo1}
\tilde {\omega}(f,r)=\Big(\frac{1}{|\H_r^-|}\int_{\H_r^-}|f(x,t)-c_r|^p\Big)^{1\over p}.
\end{equation}
If $u$ is a solution of \eqref{pbm} then
\begin{equation}\label{ncap}
\hat {N}(u,r)=\inf_{P\in\mathcal{P}_{c_r}}  \Big(\frac{1}{r^{Q+2+2p}}\int_{\H_r^-}|u-P|^p\Big)^{1\over p},
\end{equation}
where, as usual, $Q$ denotes the homogeneous dimension defined in \eqref{hom-dim}.\\
Moreover for $0<a<b$, we define
\begin{align} \label{nro}
&\hat {N}(u,a,b)=\sup_{a\le\r\le b} \hat {N}(u,\r)\\ \label{omro}
&\tilde {\omega}(f,a,b)=\sup_{a\le\r\le b} \tilde {\omega}(f,\r)
\end{align}
We also make use of the following notation
\begin{align} \label{nbar}
&\underline {N}(r)=\hat {N}(u,\lambda r,r),\\ \label{ombar}
&\underline {\omega}(r)= \tilde {\omega}(f,\lambda^2 r,r),
\end{align}
where $\lambda \in (0,1)$.

\setcounter{equation}{0}\setcounter{theorem}{0}
\section{Preliminary results}
\label{prel-results}
We here list some general ultraparabolic estimates. Some of them are well-know from the literature, so for their proofs we will refer to source.\\
First, for $\Omega$ open set in $\rnn$, $p \in \left(1,+\infty\right)$, we define the Sobolev-Stein space
\begin{eqnarray*}
S^p(\Omega)=\lbrace u \in L^p(\Omega):\partial_{x_i} u, \partial^2_{x_i x_j}u, Yu \in L^p(\Omega), \quad i,j=1,\ldots,m \rbrace.
\end{eqnarray*}
If we set
\begin{equation*}
\Vert u \Vert_{S^p(\Omega)}^p=\Vert u \Vert_{L^p(\Omega)}^p+\sum_{i=1}^m\Vert \partial_{x_i}u \Vert_{L^p(\Omega)}^p+\sum_{i,j=1}^m\Vert \partial^2_{x_i x_j}u \Vert_{L^p(\Omega)}^p+\Vert Y u \Vert_{L^p(\Omega)}^p
\end{equation*}
we have the following local a priori estimates in $S^p(\Omega)$ for solutions to $\L u =f$ (see \cite{BramCerManf}).
\begin{theorem}(Ultraparabolic interior $L^p$-estimates)\\
Assume {\bf [H.1]} holds and let $u$ be a solution to $\L u=f$ in $\Omega$, where $\Omega$ is now a bounded open set in $\rnn$. If $\Omega_1 \subset \subset \Omega$, then we can find a constant $c$, only depending on $B$, $p$, $\Omega$ and $ \Omega_1$, such that
\begin{equation}\label{intpar}
\|u\|_{S^p(\Omega_1)}\le c(\|f\|_{L^p(\Omega)}+\|u\|_{L^p(\Omega)}).
\end{equation}
\end{theorem}

We now state a general compactness result proved in \cite{CaElPoli}.
\begin{theorem} \label{compthm}
Let $\Omega$ be an open set of $\mathbb{R}^{N+1}$ and let $u\in S^{p}(\Omega)$ be a weak solution to $\L u=f$ in $\Omega$ with $f \in L^p_{loc}(\Omega)$. Then, for every $z_0\in \Omega$ and $\rho, \sigma>0$ such that  $\H_{\rho}(z_0)$ is contained in $\Omega$ and $\sigma <\frac{\rho}{2{\bf c}_H}$, with ${\bf c}_H$ defined in \eqref{e-ps.tr.in}, we have that \\

if $1 < p < Q+2$ and $p<q<p^{\ast}$ then there exists a positive constant ${\tilde{C}_{p,q}}$ such that 
\begin{equation*}
\|u(\cdot \circ h)-u\|_{L^{q}(\H_\sigma(z_0) )} \le {\tilde{C}_{p,q}} (\|u\|_{L^{p }(\H_\rho(z_0) )}+\|f\|_{L^{p}(\H_\rho(z_0))})\|h\|^{(Q+2)(\frac{1}{q}-\frac1{p^{\ast}})}
\end{equation*}
where 
\[\frac{1}{p^\ast} = \frac{1}{p}-\frac{1}{Q+2}.\]

\end{theorem}

As a preliminary result, we state and prove the following Caccioppoli type estimate which we obtained ad hoc for our problem.
\begin{lemma}(Caccioppoli type estimate)\label{cacioppolilemma}\\
Let $P \in \mathcal{P}_{c_r}$ and let $u$ be a solution to \eqref{pbm} in $\H_r^-$.
Let $p\in (1, +\infty)$ and let $\rho, r$ such that $1 \le \r<r.$ 
Then, for $W:=(u-p)|u-P|^{\frac{p}{2}-1}$, the following estimate holds:
\begin{align*} 
&\frac{2(p-1)}{   p^2} \int _{\H^-_{\rho}} |D_m W|^2 \\
 &\quad\le \left(\frac{2}{(p-1)}\,\frac{ c_2^2}{\left(r-\r\right)^2}+\frac{2 }{p }\,c_1 \,\frac{r^{2\kappa+1}}{r-\r}\right)  \int _{\H^{-}_r} W^2 +\tilde{\omega}(f,r) \,|\H_{r}^-| \,\Big( \frac1{|\H^-_r|}\int_{\H^-_r} \eta^{2p'} W^2\Big)^{\frac1{p'}},
\end{align*}
where $c_1$, $c_2$ are dimensional constants and $p'$ is such that $\frac{1}{p}+\frac{1}{p'}=1$.
\end{lemma}

\begin{proof}
On the past cylinder $\H_r^-$, we have
\begin{eqnarray}\label{twoeq}
-\L u + f =0, \quad \textrm{and} \quad -\L P +c_r=0,
\end{eqnarray}
since $u$ is a solution to \eqref{pbm} and $P \in \mathcal{P}_{c_r}$. Taking the difference of the two equations in \eqref{twoeq} and multiplying it by $\phi:=\eta^2 w |w|^{p-2}$, with $w=u-P$, we obtain
\begin{eqnarray}\label{twoeqint}
-\int_{\H_r^-}\eta^2 w |w|^{p-2} \L w =-\int_{\H_r^-}\eta^2 w |w|^{p-2}\left(f(x)-c_r\right).
\end{eqnarray}
An integration by parts shows that 
\begin{equation}\label{parts}
\begin{split}
-\int_{\H_r^-}\eta^2 w |w|^{p-2} \L w &=\int_{\H_r^-} \langle A D_m w, D_m(\eta^2 w |w|^{p-2})\rangle-\int_{\H_r^-}\eta^2 w |w|^{p-2}Y(w)\\
&=:I_1+I_2,
\end{split}
\end{equation}
where $D_m$ denotes the gradient with respect to $x_1,\ldots,x_m$.
We now observe that  
$$D_m \phi= 2 \eta D_m\eta \, w|w|^{p-2} +\eta^2 (p-1) |w|^{p-2} D_m w$$
and therefore we can rewrite the term $I_1$ on the right-hand side of \eqref{parts} as
\begin{eqnarray*}
I_1= 2\int_{\H_r^-} \langle A D_m w , D_m\eta\rangle \eta w|w|^{p-2} +(p-1)\int_{\H_r^-} \eta^2  |w|^{p-2} \langle A D_m w, D_m w\rangle .
\end{eqnarray*}
Taking advantage of $D_m W=\frac{p}2 |w|^{ \frac{p}2-1} D_m w$, for $W=w|w|^{\frac{p}{2}-1}$, the previous equation rewrites as
\begin{eqnarray}\label{stimaI1}
I_1= \frac{4(p-1)}{ p^2} \int_{\H_r^-} \eta^2 \langle A D_m W, D_m W\rangle + \frac4{p} \int_{\H_r^-} \eta  W\, \langle A D_m W, D_m \eta \rangle   .
\end{eqnarray}
We now take care of the term $I_2$ in \eqref{parts}. We first notice that
\begin{eqnarray*}
Y(W)=\frac{p}{2}|w|^{\frac{p}{2}-1}Y(w),
\end{eqnarray*}
which, together with the divergence theorem and the identity
\begin{eqnarray*}
Y(W^2\eta^2)=2\eta W^2 Y(\eta)+2\eta^2 W Y(W),
\end{eqnarray*}
yields
\begin{eqnarray}\label{stimaI2}
I_2=\frac{2}{p}\int_{\H_r^-}\eta W^2 Y(\eta).
\end{eqnarray}
Thus, combining \eqref{stimaI1} and \eqref{stimaI2}, we can rewrite identity \eqref{twoeqint} as
\begin{equation*}\label{eqzero}
\begin{split}
0&=\frac{4(p-1)}{ p^2} \int_{\H_r^-} \eta^2 \langle A D_m W, D_m W\rangle + \frac4{p} \int_{\H_r^-} \eta  W\, \langle A D_m W, D_m \eta\rangle\\
&\quad \quad+\frac{2}{p}\int_{\H_r^-}\eta W^2 Y(\eta)+\int_{\H_r^-}\eta^2 w |w|^{p-2}\left(f(x)-c_r\right).
\end{split}
\end{equation*}
Now, setting $\varepsilon=\frac{p-1}{2p}$ and using the estimate
\begin{eqnarray*}
\eta \, |W|\, |\langle A D_m W, D_m \eta\rangle| \le \varepsilon \eta^2 \langle A D_m W , D_m W\rangle+ \frac{W^2}{ 4\varepsilon} \langle A D_m \eta , D_m \eta \rangle,
\end{eqnarray*}
we finally obtain

\begin{equation}\label{nonalligned}
\begin{split}
&\frac{2(p-1)}{ p^2} \int_{\H_r^-} \eta^2 \langle A D_m W, D_m W\rangle\\
 &\quad \le \frac{2}{(p-1)} \int_{\H_r^-} W^2 \langle A D_m \eta , D_m \eta \rangle +\frac{2}{p}\int_{\H_r^-}   W^2 \eta |Y(\eta)|+\int_{\H_r^-} |f-c_r| \eta^2 |W|^{\frac{ 2(p-1)}{p}}\\
&\quad \leq \frac{2}{(p-1)} \int_{\H_r^-} W^2 \langle A D_m \eta , D_m \eta \rangle +\frac{2}{p}\int_{\H_r^-}   W^2 \eta |Y(\eta)|\\
&\qquad \qquad \qquad\qquad \qquad\qquad+|\H_r^-|\, \tilde{\omega}(f,r)\,\left(\frac{1}{|\H_r^-|}\int_{\H_r^-} \eta^{2 p' }|W|^{\frac{ 2p'(p-1)}{p}}	\right)^{\frac{1}{p'}}\\
&\quad\leq \frac{2}{(p-1)} \int_{\H_r^-} W^2 \langle A D_m \eta , D_m \eta \rangle +\frac{2}{p}\int_{\H_r^-}   W^2 \eta |Y(\eta)|\\
&\qquad \qquad \qquad\qquad \qquad\qquad+|\H_r^-|\, \tilde{\omega}(f,r)\,\left(\frac{1}{|\H_r^-|}\int_{\H_r^-} \eta^{2 p' }|W|^{2}	\right)^{\frac{1}{p'}}.
\end{split}
\end{equation}
where $p'$ is such that $\frac{1}{p}+\frac{1}{p'}=1$.
The thesis follows by making a suitable choice of the function $\eta$ in \eqref{nonalligned}. More precisely, we set
\begin{eqnarray*}
\eta(x,t)=\chi\left(\Vert (x,0) \Vert_K\right)\chi_t(t)
\end{eqnarray*}
where $\chi \in C^\infty([0,+\infty))$ is the cut-off function defined by
\begin{equation*}\label{chi}
\chi(s)=\left\{ \begin{array}{ll}
0,\quad &\textrm{if $s \geq r$},\\
1,\quad &\textrm{if $0\leq s \leq \r$},
\end{array} \right. \quad |\chi'|\leq \frac{2}{r-\r},
\end{equation*}
and $\chi_t \in C^\infty((-\infty,0]$ is defined by
\begin{equation*}\label{chit}
\chi_t(s)=\left\{ \begin{array}{ll}
0,\quad &\textrm{if $s \leq -r^2$},\\
1,\quad &\textrm{if $-\r^2\leq s \leq 0$},
\end{array} \right. \quad |\chi'_t|\leq \frac{2}{r-\r},
\end{equation*}
with $\frac{r}{2} \leq \r < r $. We observe that
\begin{equation*}
		\label{ineq}
		|Y \eta| \le c_1 \,\frac{r^{2\kappa+1}}{r-\r},
		\hspace{8mm}
		|\p_{x_j} \eta| \le \frac{c_2}{r-\r}
		\hspace{2mm} \text{for } j = 1, \ldots, m_0,
	\end{equation*}
	where $c_1$ and $c_2$ are dimensional constants. Then, accordingly to \eqref{nonalligned}, we finally obtain
	\begin{equation*}
\begin{split}
&\frac{2(p-1)}{ p^2} \int_{\H_\r^-}  | D_m W|^2\\
&\quad\leq \frac{2}{(p-1)}\frac{c_2^2}{(r-\r)^2} \int_{\H_r^-} W^2  +\frac{2}{p}\,c_1 \,\frac{r^{2\kappa+1}}{r-\r}\int_{\H_r^-}   W^2 +|\H_r^-|\, \tilde{\omega}(f,r)\,\left(\frac{1}{|\H_r^-|}\int_{\H_r^-} \eta^{2 p' }|W|^{2}	\right)^{\frac{1}{p'}}
\end{split}
\end{equation*}
and this concludes the proof.
\end{proof}

\setcounter{equation}{0}\setcounter{theorem}{0}
\section{Pointwise estimates for the Kolmogorov equation}
\label{mainsection}
This Section is the core of the paper and it is devoted to prove our main result, Theorem \ref{mainthm}. Since the proof is rather convoluted, we have decomposed it in intermediate results proved in Section \ref{prel-est}, which will be combined in Section \ref{sectionproof} in order to give a simpler proof of Theorem \ref{mainthm}.

\subsection{Preliminary estimates}
\label{prel-est}
The following result is a useful tool in order to prove Lemma \ref{large}.
\begin{lemma}\label{lcil}
The following statements hold:
\begin{itemize}
\item[(i)]there exists a constant $C_2 = C_2(p,Q) > 0$ s.t. for every polynomial $P\in \tilde{\P}$,  for any $r \geq 1$ it holds
\begin{equation*}
\left( \frac{1}{r^{Q+2+2p}}\int_{\H_r^-}|P|^p\right)^{\ip}\leq C_2 \left( \int_{\H_1^-} |P|^p\right)^\ip;
\end{equation*}
\item[(ii)] there exists a constant $\tilde{C}_2 = \tilde{C}_2(p,Q) > 0$ s.t. for every polynomial $P\in \tilde{\P}$,  for any $r < 1$ it holds
\begin{equation*}
\left( \int_{\H_1^-}|P|^p\right)^{\ip}\leq \tilde{C}_2 \left( \frac{1}{r^{Q+2+2p}} \int_{\H_r^-} |P|^p\right)^\ip.
\end{equation*}
\end{itemize}
\end{lemma}

\begin{proof}
We only carry out the proof of assertion $(i)$, since case $(ii)$ is totally analogous.
We start by writing the polynomial $P$ as $P(x,t) = a+b\cdot x +\frac{1}{2} {}^T x \cdot c \cdot x + d \ t$, where $b \in \R^m$ and $c$ is an $m \times m $ matrix. We moreover recall that $\H_r^- = B_r \times (-r^2,0)$, where $B_r = \lbrace x \in \R^N : |x|_K \leq r \rbrace$, with $|\cdot|_K$ as defined in \eqref{norm-def-form}. Then, owing to $\|(x,t)\|_K = |x|_K +|t|^{1/2}$ with in particular $|x_i|\leq r$ for $i=1\ldots m$ and $|t|\leq r^2$, there exists a constant $C>0$ s.t.
\begin{equation}\label{part1P}
\left( \frac{1}{r^{Q+2+2p}}\int_{\H_r^-}|P|^p\right)^{\ip}\leq C \left( \frac{|a|}{r^2}+\frac{|b|}{r}+ |c| + |d| \right).
\end{equation}
On the other hand it is possible to show by contradiction that
\begin{equation}\label{part2P}
|a|+ |b|+ |c| + |d| \leq C \left(\int_{\H_1^-} |P|^p\right)^\ip.
\end{equation}
Indeed, if \eqref{part2P} is false, we have that for every $C>0$ it holds
 \begin{equation*}
|a|+ |b|+ |c| + |d| > C \left(\int_{\H_1^-} |P|^p\right)^\ip.
\end{equation*}
Putting this together with \eqref{part1P} with $r\geq 1$ we infer that constant $C$ should be both less than 1 and bigger or equal to 1, which is absurd. \\
The thesis follows by the combination of \eqref{part1P} and \eqref{part2P}, with $r \geq 1$. 
\end{proof}

We now prove the following Lemma.
\begin{lemma}(Estimates on larger cylinders)\label{large}\\
Let $u$ be solution of $\L u=f$ in $\H_R^-$ for $R>2$. 
Then for any $\r\in [1, R/2]$, there exists a positive constant $C_1=C_1(p,Q)$ s.t.
\begin{equation}\label{estilarge}
\Big( \frac{1}{\r^{Q+2+2p}}\int_{\H_\r^-}|u-P_1|^p dx \ dt\Big)^{\ip}\le C_1 \int_1^{4\r} \frac{\hat {N}(u, s)+\tilde {\omega}(f, s)}{s} \ ds,
\end{equation}
where $P_1\in \mathcal{P}_{1}$.
\end{lemma}

\begin{proof}
We start working on the left hand side of \eqref{estilarge}. Namely, for any $\r \geq 1$
\begin{align}\nonumber
&\Big( \frac{1}{\r^{Q+2+2p}}\int_{\H_\r^-}|u-P_1|^p \Big)^{\ip} \\ \label{estlem}
\leq &\Big( \frac{1}{\r^{Q+2+2p}}\int_{\H_\r^-}|u-P_\r|^p \Big)^{\ip} + \Big( \frac{1}{\r^{Q+2+2p}}\int_{\H_\r^-}|P_\r-P_1|^p \Big)^{\ip} = \hat{N}(u,\r)+I_1.
\end{align}
where in the last line we recalled \eqref{ncap}, and $P_\r \in \tilde{\P_{c_\r}}$ is a polynomial realizing the infimum in the definition of $\hat{N}(u,\cdot)$ at the level $\r$.
We now estimate $I_1$ as follows
\begin{align}\label{i1est}
I_1  \leq \Big( \frac{1}{\r^{Q+2+2p}}\int_{\H_\r^-}|P_r-P_1|^p \Big)^{\ip} + \sum_{j=1}^k \Big( \frac{1}{\r^{Q+2+2p}}\int_{\H_\r^-}|P_{2^j r}-P_{2^{j-1}r}|^p \Big)^{\ip} =: I_2 + I_3,
\end{align}
where we have written $\r \geq 1$ as $\r=2^k r$ for an integer $k \geq 1$ and $r \in \left[1/2,1\right)$.
In order to control $I_2$ and $I_3$ we need to achieve a more general estimate. For an arbitrary $\gamma>1$, for any $\alpha \in [1,\gamma]$ we have that
\begin{align}\nonumber
&\Big( \frac{1}{r^{Q+2+2p}}\int_{\H_r^-}|P_{\alpha r}-P_r|^p \Big)^{\ip} \\ \nonumber
\leq &\Big( \frac{1}{r^{Q+2+2p}}\int_{\H_r^-}|u-P_r|^p \Big)^{\ip}+\Big( \frac{1}{r^{Q+2+2p}}\int_{\H_r^-}|u-P_{\alpha r}|^p \Big)^{\ip}\\ \nonumber
\leq &\Big( \frac{1}{r^{Q+2+2p}}\int_{\H_r^-}|u-P_r|^p \Big)^{\ip} + \alpha^\frac{Q+2+2p}{p} \Big( \frac{1}{(\alpha r)^{Q+2+2p}} \int_{\H_{\alpha r}^-}|u-P_{\alpha r}|^p \Big)^{\ip}\\ \nonumber
\leq& \alpha^\frac{Q+2+2p}{p} (\hat{N}(u,r)+\hat{N}(u,\alpha r))\\ \label{nalpha}
\leq& \gamma^\frac{Q+2+2p}{p} (\hat{N}(u,r)+\hat{N}(u,\alpha r)).
\end{align}

We take care of $I_2$ choosing $\alpha r = 1$, for $r \in \left[\frac{1}{\gamma},1\right)$ in \eqref{nalpha}. Namely, applying both case $(i)$ and $(ii)$ from Lemma \ref{lcil}, we get
\begin{align}\nonumber
I_2 &:= \Big( \frac{1}{\r^{Q+2+2p}}\int_{\H_\r^-}|P_r-P_1|^p \Big)^{\ip}\leq C_2 \Big( \int_{\H_1^-}|P_1-P_r|^p \Big)^{\ip}\\ \label{i2est}
&\leq C \Big( \frac{1}{r^{Q+2+2p}}\int_{\H_r^-}|P_1-P_r|^p \Big)^{\ip} \leq \gamma^\frac{Q+2+2p}{p} (\hat{N}(u,r)+\hat{N}(u,1)).
\end{align}
Exploiting again \eqref{nalpha} together with case $(i)$ from Lemma \ref{lcil}, we infer that for every $\r \geq 1$, which we write as $\r =2^k r$ with $r \in \left[\frac{1}{\gamma},1\right)$, there holds
\begin{align} \label{i3est}
I_3 \leq C \sum_{j=1}^k \hat{N}(u,2^jr).
\end{align}
Now, collecting bounds \eqref{i1est}, \eqref{i2est} and \eqref{i3est}, we have that \eqref{estlem} reads
\begin{align}\nonumber
&\Big( \frac{1}{\r^{Q+2+2p}}\int_{\H_\r^-}|u-P_1|^p \Big)^{\ip} \\ \label{semifinalest}
\leq & C \left(\hat{N}(u,1) + \hat{N}(u,r) + \sum_{j=1}^k \hat{N}(u,2^jr)\right),
\end{align}
where $C=C(p,Q,\gamma)$ is a positive constant.
We now want to estimate the right hand side of \eqref{semifinalest}, in particular for any $\gamma >1$ and for $\alpha \in [1, \gamma]$ it follows that
\begin{align}\nonumber
\hat{N}(u,\alpha r) &\leq  \Big( \frac{1}{(\alpha r)^{Q+2+2p}} \int_{\H_{\alpha r}^-}|u-P_{\alpha r}|^p \Big)^{\ip}\\ \nonumber
&\leq  \Big( \frac{1}{(\gamma r)^{Q+2+2p}} \int_{\H_{\gamma r}^-}|u-P_{\gamma r}|^p \Big)^{\ip} {\left(\frac{\gamma}{\alpha}\right)}^{\frac{Q+2+2p}{p}}
+  \Big( \frac{1}{(\alpha r)^{Q+2+2p}} \int_{\H_{\alpha r}^-}|P_{\gamma r}-P_{\alpha r}|^p \Big)^{\ip}\\ \label{Nalpha}
&\leq  \gamma^\frac{Q+2+2p}{p}\hat{N}(u,\gamma r) + C_2|c_{\alpha r}-c_{\gamma r}|\left( \int_{\H_1^-}|P_*|^p\right)^\ip
\end{align}
where in the last line we used case $(i)$ of Lemma \ref{lcil} and we introduced $P_*$ as a solution to equation $\L P_* =1$.
In particular, from \eqref{mmo1}, we obtain
\begin{align} \nonumber
|c_{\alpha r}-c_{\gamma r}| 
&= \left( \frac{1}{|\H_{\alpha r}^-|}\int_{\H_{\alpha r}^-}|c_{\alpha r}-c_{\gamma r}|^p\right)^\ip\\ \nonumber
&\leq\left( \frac{1}{|\H_{\alpha r}^-|}\int_{\H_{\alpha r}^-}|f-c_{\alpha r}|^p\right)^\ip +
\gamma^\frac{Q+2}{p}\left( \frac{1}{|\H_{\gamma r}^-|}\int_{\H_{\gamma r}^-}|f-c_{\gamma r}|^p\right)^\ip \\ \label{cest}
&\leq \tilde{\omega}(f,\alpha r) + \gamma^\frac{Q+2}{p}\tilde{\omega}(f,\gamma r)\leq 2\gamma^\frac{Q+2}{p}\tilde{\omega}(f,\gamma r). 
\end{align}
Thus, combining \eqref{cest} with \eqref{Nalpha} we infer that for any $\gamma>1$ there esists a positive constant $C_\gamma= C_\gamma(p,Q,\gamma)$ s.t. for any $\alpha \in [1,\gamma]$ it holds
\begin{align} \nonumber
&\hat{N}(u,\alpha r) \leq C_\gamma \left(\hat{N}(u,\gamma r)+\tilde{\omega}(f,\gamma r)\right),\\ \label{nmon}
& \tilde{\omega}(f,\alpha r)\leq \ C_\gamma \tilde{\omega}(f,\gamma r).
\end{align}
Eventually, putting together \eqref{semifinalest} and \eqref{nmon} and choosing $\gamma =2$ we finally obtain
\begin{align*}
\Big( \frac{1}{\r^{Q+2+2p}}\int_{\H_\r^-}|u-P_1|^p \Big)^{\ip}
&\leq 3C \sum_{j=1}^k \left(\hat{N}(u,2^{j+1} r)+\tilde{\omega}(f,2^{j+1}r)\right)\\
&\leq 6C \sum_{j=1}^k \dfrac{\hat{N}(u,2^{j+1} r)+\tilde{\omega}(f,2^{j+1}r)}{2^{j+1}}\left(2^{j+2} r-2^{j+1} r\right)\\
&\leq 6C \int_1^{4\r}\dfrac{\hat{N}(u,s)+\tilde{\omega}(f,s)}{s}ds.
\end{align*}
\end{proof}

\noindent As a consequence, we can prove the dacay estimate below.

\begin{proposition}(Basic decay estimate)\label{decayprop}\\
Given $p\in(1,+\infty)$, there exist constants $C_0 = C_0(p,Q)>0 \textrm{ and } \lambda =\lambda(p,Q), \mu=\mu(p,Q) \in (0,1)$ such that for every function $u$ and $f$ satisfying  \eqref{pbm},  $\forall r\in(0,1]$, the following estimates hold
\begin{equation} \label{decayest}
\hat {N}(u, \lambda^2 r,\lambda r)<\mu \ \hat {N}(u,\lambda r,r)   \qquad{\text or} \qquad \hat {N}(u, \lambda^2 r,\lambda r)<C_0 \ \tilde {\omega}(f,\lambda^2 r,r).
\end{equation}
\end{proposition}

\begin{proof}
The proof is carried out by contradiction. Namely, if \eqref{decayest} is not true, we can find the sequences $C_k\to \infty, r_k\in(0,1],\lambda_k\to 0$ and $\mu_k\to 1$ such that
 \begin{align}\label{contra}
\hat {N}(u_k, \lambda_k^2 r_k,\lambda_k r_k)&\ge \mu_k \hat {N}(u_k,\lambda_k r_k,r_k)\\
\hat {N}(u_k, \lambda_k^2 r_k,\lambda_k r_k)&\ge C_k \tilde {\omega}(f_k, \lambda_k^2 r_k, r_k),
\end{align}
where $(f_k)_k$ and $(u_k)_k$ satisfy \eqref{pbm}.
Let us consider $\r_k\in[\lambda_k^2 r_k, \lambda_k r_k]$ such that, according to \eqref{nro} 
\begin{align}\label{defeps}
\hat {N}(u_k, \lambda_k^2 r_k,\lambda_k r_k)&=\hat {N}(u_k, \r_k)=\colon \e_k.
\end{align}
Moreover, owing to \eqref{e-dilations}, we define the rescaled functions
 \begin{align*}
v_k(x,t)&= \frac{u_k(\delta_{\r_k}( x,t))}{\r_k^2}
\end{align*}
and 
 \begin{align}\label{defw}
w_k(x,t)&= \frac{u_k(\delta_{\r_k}( x,t))-P_k(\delta_{\r_k}( x,t))}{\e_k \r_k^2}
\end{align}
where  $P_k \in \P_{c \r_k}$ is the homogeneous polynomial realizing the infimum at the level $\r_k$.

Now we want to control $w_k$ in order to pass to the limit. We first notice that
\begin{equation}\label{inf1}
\inf_{P\in \P} \Big(\int_{\H_1^-} |w_k -P|^p \Big)^{\frac1p}=1. 
\end{equation}

Indeed, first exploiting the definition of $w_k$ in \eqref{defw} and then using the change of variables $y = \delta_{\r_k}^0(x), \ s= \r_k^2 t$, owing to \eqref{fam-dil-space}, we infer
\begin{align*}
&\inf_{P\in \P} \Big(\int_{\H_1^-} |w_k -P|^p \Big)^{\frac1p}\\
=&\inf_{P\in\P} \Big(\int_{\H_1^-} \left| \dfrac{u_k(\delta^0_{\r_k}(x),\r_k^2 t)-P_k(\delta^0_{\r_k}(x),\r_k^2 t)-\e_k \r_k^2 P(x,t)}{\e_k \r_k^2}\right|^p dx\ dt\Big)^{1/p}\\
=&\frac{1}{\e_k}\inf_{P\in\P} \Big( \frac{1}{\r_k^{Q+2+2p}}\int_{\H_{\r_k}^-} \left| {u_k(y,s)-P_k(y,s)-\e_k \r_k^2 P\left(\delta^0_\frac{1}{\r_k}(y),\frac{1}{\r_k^2}s \right)}\right|^p dy\ ds\Big)^{1/p}.
\end{align*}
Now, since $\L(P_k+\e_k \r_k^2 P)= \L P_k +\e_k \r_k^2 \L P = c_{\r_k}$ by \eqref{p} and \eqref{pc}, the identity \eqref{inf1} follows from \eqref{defeps}.

In addition it holds
 \begin{align*}
\hat {N}(v_k, 1)&=\Big(\frac1{\r_k^{Q+2+2p}}\int_{\H_{\r_k}^-} |u_k -P_k|^p \Big)^{\frac1p}.
\end{align*}

We now apply Lemma \ref{large} to $v_k$, for $s \in \left[1,\frac{r_k}{2\r_k}\right]$
\begin{align}\nonumber
\Big(\frac1{(s{\r_k)}^{Q+2+2p}}\int_{\H_{s\r_k}^-} |u_k -P_k|^p \Big)^{\frac1p}&=\Big(\frac1{s^{Q+2+2p}}\int_{\H_{s}^-} \left| \dfrac{u_k(\delta_{\r_k}(y,s))-P_k(\delta_{\r_k}(y,s))}{ \r_k^2}\right|^p\Big)^{\frac1p}\\ \nonumber
&\le C_1\int_{1}^{4s}\frac{\hat {N}(v_k, \tau )+ \tilde {\omega}(g_k, \tau )}{\tau} d\tau\\ \label{intv1}
&\le C_1\int_{1}^{4s}\frac{\hat {N}(u_k, \tau \r_k)+ \tilde {\omega}(f_k, \tau \r_k)}{\tau} d\tau,
\end{align}
where in the second line we defined $g_k(x,t)=f_k(\delta_{\r_k}(x,t))$ and in the third line we used the identities $\hat{N}(v_k,s)=\hat{N}(u_k,s\r_k)$ and $\tilde{\omega}(g_k,s)=\tilde{\omega}(f_k,s\r_k)$. As a consequence, for $s \in \left[1,\frac{r_k}{2\r_k}\right]$, the following holds 
\begin{align}\nonumber
\Big(\frac1{s^{Q+2+2p}}\int_{\H_{s}^-} |w_k |^p \Big)^{\frac1p}&\le \frac{C_1}{\e_k}\int_{1}^{4s}\frac{\hat {N}(u_k, \tau \r_k)+ \tilde {\omega}(f_k, \tau \r_k)}{\tau} d\tau\\ \label{intv}
&\le \frac{C_1}{\e_k}\int_1^{4s}\frac{\hat {N}(u_k, \lambda_k^2 r_k, r_k)+\tilde {\omega}(f_k, \lambda_k^2 r_k, r_k)}{\tau} d\tau.
\end{align}

\noindent On the other hand, combining \eqref{contra} with \eqref{defeps}, we obtain
\begin{equation*}
\hat {N}(u_k, \lambda_k^2 r_k, r_k)\le \frac{\e_k}{\mu_k}
\end{equation*}
and
\begin{equation}\label{ome}
\tilde{\omega}(f_k, \lambda_k^2 r_k, r_k)\le \frac{\e_k}{C_k}.
\end{equation}
These two bounds together with \eqref{intv}, yield to
\begin{equation}\label{log-bound}
\Big(\frac1{s^{Q+2+2p}}\int_{\H_{s}^-} |w_k |^p \Big)^{\frac1p}\le C_2 \ln 4s
\end{equation}
where $s\in \left[1,\frac{r_k}{2\r_k}\right]$ and $C_2$ is a positive constant depending on $C_1, C_k$ and $\mu_k$.
%
Now, according to the dilation invariance of $\L$ with respect to $\delta_r$ (see \eqref{Ginv}) and \eqref{ome}, we find
\begin{equation}\label{sto}
\Big(\frac1{|\H_{s}^-|}\int_{\H_{s}^-} |\L w_k |^p \Big)^{\frac1p}\le \frac1{\e_k}\tilde{\omega}(f_k, s \r_k)\le \frac1{\e_k}\tilde{\omega}(f_k, \lambda_k^2 r_k, r_k)\le \frac1{C_k}\to 0
\end{equation}

The contradiction follows from passing to the limit. In order to do so, we need a compactness argument.
Applying Lemma \ref{cacioppolilemma} to $W_k=w_k|w_k|^{\frac{p}{2}-1} $, we obtain that for every $R \in \left[1,\frac{r_k}{\r_k}\right)$
\begin{align*} 
&\frac{2(p-1)}{   p^2} \int _{\H^-_{R}} |D_m W_k|^2 \\
&\quad\le \left(\frac{2}{(p-1)}\,\frac{ c_2^2}{\left(r-\r\right)^2}+\frac{2 }{p }\,c_1 \,\frac{r^{2\kappa+1}}{r-\r}\right)  \int _{\H^{-}_r} W_k^2 +\tilde{\omega}(f,r) \,|\H_{r}^-| \,\Big( \frac1{|\H^-_r|}\int_{\H^-_r} \eta^{2p'} W_k^2\Big)^{\frac1{p'}},
\end{align*}
for $r=\frac{r_k}{\r_k}$. As a consequence, for every $R \in \left(0,\frac{r_k}{\r_k}\right)$, we have
\begin{eqnarray*}
\parallel W_k \parallel_{S^2(\H_R^-)}\leq C_R.
\end{eqnarray*}
Thus, we can extract a non-relabelled subsequence $W_k$ such that
\begin{align*}
W_k \rightharpoonup W_\infty=w_\infty|w_\infty|^{\frac{p}{2}-1}\qquad \textrm{weakly in $S^2_{loc}(\H_R^-)$},
\end{align*}
where we denoted by $w_\infty$ the limit of the sequence $w_k$. 
Moreover, from the compact embedding provided by Theorem \ref{compthm} it follows that
\begin{align*}
W_k \to W_\infty=w_\infty|w_\infty|^{\frac{p}{2}-1}\qquad \textrm{in $L^2_{loc}(\H_R^-)$}.
\end{align*}
We now observe that
\begin{eqnarray*}
\parallel w_k \parallel^p_{L^p(\H_R^-)}=\parallel W_k \parallel^2_{L^2(\H_R^-)}, \qquad \parallel w_\infty \parallel^p_{L^p(\H_R^-)}=\parallel W_\infty \parallel^2_{L^2(\H_R^-)},\qquad
\end{eqnarray*}
and therefore we have the following convergence result for every $R \in \left(0,\frac{r_k}{\r_k}\right)$
\begin{eqnarray*}
w_k \to w_\infty \qquad \textrm{in $L^p_{loc}(\H_r^-)$}.
\end{eqnarray*}

In particular, from \eqref{inf1}, $w_\infty$ satisfies
\begin{equation}\label{one}
\inf_{P\in\mathcal{P}} \Big(\int_{\H_1^-} |w_\infty -P|^p \Big)^{\frac1p}=1.
\end{equation}
Similarly, according to \eqref{log-bound},
\begin{equation*}
\Big(\frac1{s^{Q+2+2p}}\int_{\H_{s}^-} |w_\infty |^p \Big)^{\frac1p}\le C_2 \ln 4s.
\end{equation*}
Hence, $w_\infty$ is a function that grows quadratically in space and linearly in time up to a logarithmic correction. Moreover, in virtue of \eqref{sto}, it follows that $w_\infty$ a. e. belongs to $\P$.
This contradicts \eqref{one} and therefore concludes the proof.
\end{proof}

We now establish a sort of monotonicity result for $\underline{N}$ and $\underline{\omega}$, defined in \eqref{nbar} and \eqref{ombar}, respectively.
\begin{proposition}(Dini estimate)\label{dini-proposition}\\
Let $N:(0,1] \rightarrow [0,+\infty)$, $\omega:(0,1] \rightarrow [0,+\infty)$ be two functions that satisfy
\begin{equation} \label{dini-decay}
\forall r \in (0,1], \qquad {N}(\lambda r) < \mu \, {N}(r)   \qquad{\text or} \qquad {N}(\lambda r) <\underline{C} \, {\omega}(r),
\end{equation}
and
\begin{equation}\label{dini-mon}
\forall r \in (0,1] \quad \forall \alpha \in [\lambda,1],\qquad \left\{ \begin{array}{ll}
{N}(\alpha r)\leq \underline{C} \, ({N}(r)+{\omega}(r)),\\
{\omega}(\alpha r)\leq \underline{C} \, {\omega}(r)
\end{array} \right.
\end{equation}
for some constants $\underline{C} >0$ and for $\lambda, \mu \in (0,1)$. Moreover, we assume that ${\omega}$ is Dini. Then for every $\r \in (0,\frac{\lambda}{4})$ and for $\beta=\frac{\ln \mu}{\ln \lambda}$ we have
\begin{equation}\label{dini-est}
\begin{split}
&\int_0^{4\r}\frac{{N}(r)}{r}dr \leq \\ &\qquad\underline{C} \, \frac{1}{\beta}\left\lbrace \left(\frac{{4\r}}{\lambda}\right)^{\beta} \left({N}(1)+{\omega}(1)\right) +\underline{C}'\left(\int_0^{4\r} \frac{{\omega}(r)}{r}dr + \r^{\beta}\int^1_{4\r}\frac{{\omega}(r)}{r^{1+\beta}}dr \right)\right\rbrace.
\end{split}
\end{equation}
where $\underline{C}'=\underline{C}'( \lambda, \mu)=\frac{1}{\mu}\frac{1}{(1-\lambda)\lambda^\beta}$.
\end{proposition}

\begin{proof}
We first prove that for all $r \in (0,\lambda]$, we have
\begin{eqnarray}\label{max-est}
{N}(r)\leq \max \left(\underline{C}_1 r^{\beta},\underline{C} \, \frac{1}{\mu} \, r^{\beta}\sup_{\r \in [r,\lambda]}\frac{{\omega}(\r)}{\r^{\beta}} \right),
\end{eqnarray}
where $\underline{C}_1$ is given by
\begin{eqnarray}\label{def-C_1}
\underline{C}_1=\underline{C} \, \lambda^{-\beta} \, \left({N}(1)+{\omega}(1)\right).
\end{eqnarray}

If $r \leq \lambda$, we write it as $r=\lambda^k r_1$ with $k \geq 1$ and $r_1 \in (\lambda,1]$. Then, taking advantage of \eqref{dini-decay}, we infer
\begin{align}\label{decay-1} \nonumber
{N}(r) &\leq \max \left(\underline{C} \,{\omega}\left(\frac{r}{\lambda}\right),\mu \, {N}\left(\frac{r}{\lambda}\right)\right)\\
&\leq \max \left(\underline{C} \,{\omega}\left(\frac{r}{\lambda}\right),\underline{C} \, \mu \,{\omega}\left(\frac{r}{\lambda^2}\right), \mu^2 \, {N}\left(\frac{r}{\lambda^2}\right)\right)\\ \nonumber
&\leq \max \left(\underline{C} \,{\omega}\left(\frac{r}{\lambda}\right),\underline{C} \, \mu \,{\omega}\left(\frac{r}{\lambda^2}\right),\underline{C} \, \mu^2 \,{\omega}\left(\frac{r}{\lambda^3}\right), \ldots, \underline{C} \, \mu^{k-2} \,{\omega}\left(\frac{r}{\lambda^{k-1}}\right) ,\mu^k \, {N}\left(\frac{r}{\lambda^k}\right)\right).
\end{align}
Now, if we set $\beta=\frac{\ln \mu}{\ln \lambda}$ and $\r=\frac{r}{\lambda^{j+1}}$ for $j=0,\ldots,k-2$, we deduce
\begin{equation}\label{decay-2}
\begin{split}
\mu^j \, {\omega}\left(\frac{r}{\lambda^{j+1}}\right) = e^{j \ln \mu}{\omega}(\r)=\mu^{-1} \, e^{\ln(r/\r)\beta} \, {\omega}(\r)=\mu^{-1} \frac{{\omega}(\r)}{\r^{\beta}}r^{\beta}.
\end{split}
\end{equation}
On the other hand, according to \eqref{dini-mon} and \eqref{def-C_1}, we have
\begin{equation}\label{decay-3}
\begin{split}
\mu^k \, {N}\left(\frac{r}{\lambda^k}\right) \leq \mu^k\underline{C} ({N}(1)+{\omega}(1))=\underline{C}_1 \, \mu^k \lambda^{\beta} \leq \underline{C}_1 \, \mu^k r_1^{\beta} =\underline{C}_1 \, \mu^k \left(\frac{r}{\lambda^k}\right)^{\beta} \leq \underline{C}_1 \, r^{\beta}.
\end{split}
\end{equation}
Finally, using estimates \eqref{decay-2} and \eqref{decay-3} in \eqref{decay-1}, we get \eqref{max-est}.

We now want to estimate $\sup_{\r \in [r,\lambda]}\frac{{\omega}(\r)}{\r^{\beta}}$. To this end, for some $\r_0 \in [r,\lambda]$, we write
\begin{equation*}
\begin{split}
\sup_{\r \in [r,\lambda]}\frac{{\omega}(\r)}{\r^{\beta}}&=\frac{{\omega}(\r_0)}{\r_0^\beta} \\
&\leq \frac{1}{\r_0^\beta}\frac{1}{t \r_0}\int_{\r_0}^{\r_0+t \r_0}\underline{C} \, {\omega}(\r)d\r\\
&\leq \frac{\underline{C}}{t \lambda^{1+\beta}}\int_{\r_0}^{\r_0/\lambda}\frac{{\omega}(\r)}{\r^{1+\beta}}d\r\\
&\leq \underline{C}_2 \, \int_{r}^1 \frac{{\omega}(\r)}{\r^{1+\beta}}d\r,
\end{split}
\end{equation*}
where in the second line we have used the monotonicity of ${\omega}$ according to \eqref{dini-mon} and the constants $t$ and $\underline{C}_2$ appearing in the second and forth line are equal to $(1-\lambda)/\lambda$ and $\underline{C}/((1-\lambda)\lambda^\beta)$ respectively.

Combining the previous inequality with \eqref{max-est} and setting $\underline{C}':=\frac{1}{\mu} \, \underline{C}_2$, we obtain for any $\r \in \left(0,\frac{\lambda}{4}\right)$
\begin{equation}\label{N/r-est}
\int_0^{4\r} \frac{{N}(r)}{r}dr \leq \underline{C}_1 \int_0^{4\r} r^{\beta-1}dr + \underline{C} \, \underline{C}' \, J,
\end{equation}
with
\begin{equation}\label{def-J}
\begin{split}
J:&=	\int_0^{4\r} r^{\beta-1} dr\left(\int_r^1 \frac{{\omega}(\tau)}{\tau^{1+\beta}}d\tau \right)\\ 
&\leq\int_0^{4\r} \frac{r^\beta}{\beta} \frac{ {\omega}(r)}{r^{1+\beta}}d r + \left[ \frac{r^\beta}{\beta}\left(\int_r^1 \frac{{\omega}(\tau)}{\tau^{1+\beta}}d\tau \right) \right]_0^{4\r}\\
&=\frac{1}{\beta}\int_0^{4\r} \frac{{\omega}(r)}{r}dr+\frac{{(4\r)}^\beta}{\beta}\left(\int^1_{4\r} \frac{ {\omega}(\tau)}{\tau^{1+\beta}}d\tau \right),
\end{split}
\end{equation}
where in the second line we have integrated by parts and in the third we have applied the dominated convergence theorem.

Inequality \eqref{N/r-est}, together with \eqref{def-J} and the definition of $\underline{C}_1$ in \eqref{def-C_1}, yields
\begin{equation*}
\begin{split}
\int_0^{4\r} \frac{{N}(r)}{r}dr \leq \underline{C}\left(\frac{{4\r}}{\lambda}\right)^\beta \frac{1}{\beta}({N}(1)+{\omega}(1))+ \underline{C} \underline{C}'\left(\frac{1}{\beta}\int_0^{4\r} \frac{{\omega}(r)}{r}dr+\frac{({4\r})^\beta}{\beta}\left(\int^1_{4\r} \frac{{\omega}(\tau)}{\tau^{1+\beta}}d\tau \right)\right),
\end{split}
\end{equation*}
which concludes the proof. 
\end{proof}

\begin{remark}
We observe that hypothesis \eqref{dini-mon} could be substituted by \eqref{nmon} and therefore owing to Proposition \ref{decayprop}, the previous result Proposition \ref{dini-proposition} holds in parti-cular for $\hat{N}$ and $\tilde{\omega}$.
\end{remark}

\begin{remark} 
We notice that:
\begin{itemize}
\item[1.] the quantities $\underline{N}$ and $\underline{\omega}$ defined in \eqref{nbar} and \eqref{ombar} satisfy \eqref{dini-decay} in virtue of Proposition \ref{decayprop}. Moreover, \eqref{nmon} with $\gamma=\frac{1}{\lambda}$ implies that $\underline{N}$ and $\underline{\omega}$ also satisfy \eqref{dini-mon};
\item[2.] we chose the extremes of integration in order to combine effortlessly this result with the following Lemma \ref{leminterm}.
\end{itemize}
\end{remark}

We now focus on the following result, which differs from Lemma \ref{large} in the choice of the polynomial and of $\r$. More precisely, in Lemma \ref{large}, we derive an estimate on large cylinders, while we here consider smaller radii.

\begin{lemma}(Estimates on smaller cylinders)\label{leminterm}\\
If $u$ is defined in $\H_1^-$, then there exist a polynomial $P_0 \in \tilde{\P}$ such that for every $\r \in (0,\frac{1}{4})$, we have
\begin{equation*}
\left(\frac{1}{\r^{Q+2+2p}}\int_{\H_\r^-}|u-P_0|^p\right)^{\ip}\leq C_1 \int_0^{4\r}\frac{\hat{N}(r)+\tilde{\omega}(f,r)}{r}dr.
\end{equation*}
\end{lemma}

\begin{proof}
We suppose that $u$ is a solution to $\L u =f$ in $\H_1^-$. Applying Lemma \ref{large} to a rescaled function 
\begin{equation*}
v(x,t) = \frac{u(\delta_{r}(x,t))}{r^2} 
\end{equation*}
it follows that for $r\leq \frac{1}{4\gamma}$, with $\gamma \geq 1$
\begin{equation*}
\left(\dfrac{1}{\gamma^{Q+2+2p}}\int_{\H_\gamma^-}|v-P^v|^p\right)^\ip\leq C_1\int_1^{4\gamma} \frac{\hat{N}(v,s)+\tilde{\omega}(f,s)}{s}ds
\end{equation*}
where $P^v$ realises the infimum in the definition of $\hat{N}(v,1)$. Now performing a change of variables with $\r = \gamma r$, we infer 
\begin{equation}\label{lemlim}
\left(\dfrac{1}{\r^{Q+2+2p}}\int_{\H_\r^-}|u-P_r|^p\right)^\ip\leq C_1\int_r^{4\r} \frac{\hat{N}(u,s)+\tilde{\omega}(f,s)}{s}ds,
\end{equation}
where we notice that $P^v(x,t) = \frac{P_r(\delta_r(x,t))}{r^2}$ and $\hat{N}(v,s) = \hat{N}(u,rs)$.
Hence, fixing $\r \in (0,1/4)$, we may pass to the limit in \eqref{lemlim} for $r \to 0$. Therefore up to extracting a subsequence, we can assume that $P_r$ tends to a polynomial $P_0 \in \tilde{\P} $, namely
\begin{equation*}
\left(\dfrac{1}{\r^{Q+2+2p}}\int_{\H_\r^-}|u-P_0|^p\right)^\ip\leq C_1\int_0^{4\r} \frac{\hat{N}(u,s)+\tilde{\omega}(f,s)}{s}ds,
\end{equation*}
which concludes the proof.
\end{proof}

We notice that, in Lemma \ref{leminterm} and the upcoming Proposition, we have $P_0$ belonging to the set $\tilde{\P}$. Hence, in the proof of assertion $(iii)$ of Theorem \ref{mainthm} it is only left to show that $P_0$ belongs in particular to $\P$ (i.e. $\L P_0 =0$) in order to prove \eqref{est-thm}.

\begin{proposition}(Modulus of continuity of the solution up to second order) \label{propcont}\\
Let us assume that $\tilde{\omega}$ is Dini continuous and let us set $\beta=\ln \mu/\ln \lambda$. There exist a polynomial $P_0 \in \tilde{\P}$ and a constant $C'=C'(\underline{C},\lambda,\mu)$ such that for every $\r \in (0,\frac{\r^*}{4})$, we have
\begin{align}\nonumber
&\left(\dfrac{1}{\r^{Q+2+2p}}\int_{\H_\r^-}|u-P_0|^p\right)^\ip \\
&\quad \leq \label{estfin}
\underline{C} \, \frac{1}{\beta}\left\lbrace \left(\frac{{4\r}}{\lambda}\right)^{\beta} \left(\hat{N}(1)+\tilde{\omega}(1)\right) +\underline{C}'\left(\int_0^{4\r} \frac{\tilde{\omega}(r)}{r}dr + \r^{\beta}\int^1_{4\r}\frac{\tilde{\omega}(r)}{r^{1+\beta}}dr \right)\right\rbrace.
\end{align}
\end{proposition}

\begin{proof}
The proof simply follows from the combination of Proposition \ref{dini-proposition} with $N\equiv\hat{N}$ and $\omega\equiv \tilde{\omega}$ and Lemma \ref{leminterm}, where we remark that we reabsorbed the modulus of continuity in the right hand side of \eqref{dini-est}.
\end{proof}

\begin{remark}
We observe that Proposition \ref{propcont} holds, more generally, for two functions $N$ and $\omega$ satisfying the assumptions of Proposition \ref{dini-proposition}.
\end{remark}

\subsection{Proof of Theorem \ref{mainthm} }
\label{sectionproof}
We start proving assertion $(i)$ of Theorem \ref{mainthm}. 
We first observe that from definitions \eqref{ncap} and \eqref{ntilde}, it holds that
\begin{equation}\label{nsp}
\tilde{N}(u,r)\leq \hat{N}(u,r).
\end{equation}
The right hand side of \eqref{nsp} can be estimated combining \eqref{max-est} with \eqref{def-C_1}, which yields for $r\in (0,\lambda]$
\begin{equation*}
\underline{N}(r) \leq \underline{C}\left( \underline{N}(1)+\frac{1}{\mu}\sup_{\r\in(0,1]} \underline{\omega}(\r)\right),
\end{equation*}
where we recall that $\underline{N}$ and $\underline{\omega}$ were defined respectively in \eqref{nbar} and \eqref{ombar}.
Moreover owing to the motonicity-type estimate \eqref{nmon} for $\gamma =\frac{1}{\lambda}$ and $r\in(0,1]$
\begin{equation*}
\hat{N}(u,r) \leq C \left( \hat{N}(u,1)+\sup_{\r\in(0,1]} \tilde{\omega}(f,\r)\right),
\end{equation*}
with $C = C(\lambda,\mu, p, Q)$. On the other hand, from definition \eqref{ncap} we get
\begin{equation*}
\hat{N}(u,1) \leq C\left(\|u\|_{L^p(\H_1^-)} + \|f\|_{L^p(\H_1^-)}\right).
\end{equation*}
Therefore, combining the estimates above, we conclude the proof of statement $(i)$.
\newline
Assertion $(ii)$ follows directly from estimate \eqref{max-est}.\\
We now focus on proving $(iii)$. We first observe that, from Proposition \ref{dini-proposition}, it follows that if $\tilde{\omega}(f,\cdot)$ is Dini, then $\hat{N}(u,\cdot)$ is Dini.
Moreover we recall that we have already proved estimate \eqref{est-thm} in the case where $P_0 \in \tilde{\P}$, according to Proposition \ref{propcont} and namely to \eqref{estfin}. Furthermore, we notice that the coefficients of $P_0$ are bounded by choosing $\r = \frac{\lambda}{4}$ in \eqref{estfin}.\\
Therefore it is only left to show that $P_0$ belongs in particular to $\P$, i.e. $P_0$ satisfies equation $\L P_0 =0$.\\
To this end, we define the function
\begin{equation*}
u^\varepsilon (x,t) = \frac{u(\delta_\varepsilon(x,t))-P_0(\delta_\varepsilon(x,t))}{\varepsilon^2}
\end{equation*}
which converges in $L^p$ to a function $v\equiv 0$ by \eqref{estfin} for $\varepsilon \to 0$. Moreover from 
\begin{align*}
\L(u^\varepsilon)&=\varepsilon^2 \dfrac{\L u(\delta_\varepsilon(x,t))}{\varepsilon^2}-\varepsilon^2\dfrac{\L P_0(\delta_\varepsilon(x,t))}{\varepsilon^2} \\
&= f(\delta_\varepsilon(x,t))-\L P_0(\delta_\varepsilon(x,t))
\end{align*}
and according to \eqref{intpar}, it follows that for $\varepsilon \to 0$,
\begin{align*}
0 = \L v = f(0)-\L P_0.
\end{align*}
Since by assumption $f(0) =0$, then we have showed that $P_0$ satisfies equation $\L P_0 =0$. This concludes the proof.

\medskip
\smallskip

\noindent\textbf{Acknowledgements}\\
The authors would like to thank Prof. Sergio Polidoro and Prof. Bianca Stroffolini for suggesting the problem and for useful comments.


\end{document}